\documentclass[sn-mathphys]{sn-jnl}% Math and Physical Sciences Reference Style
\jyear{2021}%

\theoremstyle{thmstyleone}%
\newtheorem{theorem}{Theorem}%  meant for continuous numbers
\theoremstyle{thmstyleone}
\newtheorem{proposition}{Proposition}%

\theoremstyle{thmstyleone}
\newtheorem{lemma}{Lemma}

\theoremstyle{thmstyletwo}%
\newtheorem{example}{Example}%

\theoremstyle{thmstylethree}%
\newtheorem{definition}{Definition}%
\usepackage{cases}
\usepackage{amsmath, empheq}
\begin{document}

\title[Sequential Henig proper optimality conditions...]{Sequential Henig proper optimality conditions for multiobjective fractional programming problems via sequential proper subdifferential calculus }


\author*[1]{\fnm{Mohamed Bilal}\sur{ Moustaid}}\email{bilalmoh39@gmail.com}
\equalcont{These authors contributed equally to this work.}

\author[1]{\fnm{Mohamed} \sur{Laghdir }}\email{laghdirm@gmail.com}

\equalcont{These authors contributed equally to this work.}

\author[1]{\fnm{Issam}\sur{ Dali}}\email{dali.issam@gmail.com}
\equalcont{These authors contributed equally to this work.}

\author[2]{\fnm{Ahmed}\sur{ Rikouane}}\email{rikouaneah@gmail.com}
\equalcont{These authors contributed equally to this work.}

\affil*[1]{\orgdiv{Department of mathematics}, \orgname{Chouaib Doukkali University}, \orgaddress{\street{} \city{El Jadida}, \postcode{24000}, \country{Morocco}}}

\affil[2]{\orgdiv{Department of mathematics}, \orgname{Ibn Zohr University}, \orgaddress{\street{} \city{Agadir}, \postcode{80000}, \country{Morocco}}}


\abstract{In this paper, in the absence of any constraint qualifications, we  develop sequential  necessary and sufficient optimality conditions for a constrained multiobjective fractional programming problem characterizing a Henig proper efficient solution in terms of the $\epsilon$-subdifferentials and the subdifferentials of the functions. This is achieved by employing a sequential Henig subdifferential calculus rule of the sums of $m \ (m\geq 2)$ proper convex vector valued mappings with a composition of two convex vector valued mappings. In order to present an example illustrating the main results of this paper, we establish the classical optimality conditions under Moreau-Rockafellar qualification condition. Our main results are presented in the setting of reflexive Banach space in order to avoid the use of nets.}



\maketitle

\section{Introduction}

Fractional multiobjective programming  has many applications such management science, operational research, economics and information theory (see Stancu-Minasian \cite{stancu, Sta}). However, to obtain optimality conditions for a multiobjective fractional problem which is not generally convex,  we often formulate an equivalent  vector convex problem by using a parametric approach, but such vector convex program requires a regularity condition as generalized Slater constraint qualification which is not an easy task to verify and may fail to hold neither for finite nor infinite dimensional convex program.
To raise this drawbacks, many contributions have been focused on the characterizations of sequential optimality condition for a vector convex problem or a fractional optimization problem which avoid a constraint qualification (see \cite{jevv, kohli, Kim, lag1, lag2, sun, thibault}). The recent  contributions of Moustaid-Laghdir-Dali-Rikouane in \cite{lagmou6001, conf1} motivate the present work where the sequential optimality conditions are stated, via sequential subdifferential calculus, in terms of limits of sequences in approximate subdifferential or in exact subdifferential at neabry points. So, the purpose of this paper is to establish sequential optimality conditions for a fractional optimization problem without any constraint qualifications characterizing completely a Henig efficient solution. Firstly, we establish a sequential Henig subdifferential calculus rule of the sums of $m \ (m\geq 2)$ proper convex vector valued mappings with a composition of two convex vector valued mappings. This is achieved by employing a scalarization process and the epigraph of the conjugate of the sums of $m \ (m\geq 2)$ proper convex lower semicontinuous functions. Secondly, since a multiobjective fractional problem  is  transformed equivalently into a convex multiobjective problem, we  apply the sequential Henig subdifferential calculus rule to  obtain three  different kinds of sequential optimality conditions.  The first  is expressed in terms of the epigraphs of the conjugate of data functions, the second is obtained by means  of a sequence of  $\epsilon$-subdifferentials  at
a minimizer and the last kind is described, by means of the  Br{\o}ndsted-Rockafellar theorem,  in terms of the exact subdifferentials of the functions involved at nearby points to the minimizer.

This paper is organized as follows. In section 2, we recall some notions and we give some preliminary results. In section 3, we establish a sequential Henig subdifferential calculus rule of the sums of $m \ (m\geq 2)$ proper convex vector valued mappings with a composition of two convex vector valued mappings. In section 4, we provide some sequential efficient optimality conditions characterizing a Henig proper solution for a vector fractional optimization problem. In order to present an example illustrating the main result of this work, we will need to  establish the standard optimality conditions  of a multiobjective fractional problem under a constraint qualification.
\section{Preliminaries and basic definitions}
In this section, we give some definitions and preliminary results which will be used throughout this paper. Let $X,$ $Y$ and $Z$ be three real topological vector spaces and their respective continuous dual spaces $X^{*},$ $Y^{*}$ and $Z^{*}$ with duality paring denoted by $\langle .,.\rangle.$ We will use the symbol $w^{*}$ for the weak-star topology on the dual space and $\tau_{\mathbb{R}}$ for the Euclidian topology on the real line $\mathbb{R}.$ Let $Z_{+}\subset Z$ be a nontrivial convex cone and we assume in addition that $Z_{+}$ is pointed (i.e. $-Z_{+}\cap Z_{+}=\{0\}).$
The nonnegative orthant of $\mathbb{R}^{m}$ is denoted by $\mathbb{R}^{m}_{+}.$ The polar cone and strict polar cone are defined respectively as
\[ Z_+^{*}:= \{ z^* \in Z^* : \langle z^*, z\rangle\geq 0,\;\forall z\in Z_+\} \]
and
\[ (Z_+^{*})^\circ:= \{ z^* \in Z^* : \langle z^*, z\rangle> 0,\;\forall z\in Z_+\backslash\{0\}\}. \]
 On $Z$ we define for any $z_{1},z_{2}\in Z$ the following relations \begin{eqnarray*}
z_1  & \leq_{Z_+} & z_2 \iff z_2 - z_1 \in Z_+,\\[0.5mm]
z_1 & \lneqq_{Z_+} & z_2 \iff z_2 - z_1 \in Z_+\backslash \{0\}.
\end{eqnarray*}
We attach to $Z$ an abstract maximal element, denoted by $+\infty_{Z},$  with respect to $"\leq_{Z_+}"$ and we
denote $\overline{Z}=Z\cup\{+\infty_Z\}.$ Then for every $z\in Z$ one has $z\leq_{Z_+}+\infty_{Z}.$  The algebraic operations of $Z$ are extended as follows
\[z+(+\infty_Z)=(+\infty_Z)+z = +\infty_Z, \quad \forall z \in Z\cup \{+\infty_Z\}, \]
\[ \alpha \cdot (+\infty_Z)  =  +\infty_Z,  \quad \forall \alpha \geq 0. \]
 If $Z=\mathbb{R}$ and $K=\mathbb{R}_{+}$ then $\overline{\mathbb{R}}=\mathbb{R}\cup\{+\infty\},$ where $+\infty=+\infty_{\mathbb{R}}.$\\
For a given mapping $f$ : $X\longrightarrow \overline{Z},$ we denote respectively the effective domain and the epigraph of $f$ by $\mathrm{dom}f$ and $\mathrm{epi}f,$ i.e.,
\[\mathrm{dom}f:=\{x \in X : f(x) \in Z\}\]
\[\qquad\quad \mathrm{epi}f:=\{(x,z) \in X \times Z : f(x) \leq_{Z_+} z \} .\]
We say that $f$ is proper if $\mathrm{dom}f\neq \emptyset$ and $Z_{+}$-epi-closed if its epigraph is closed. Let us note for each $z^{*}\in Z_{+}^{*},$ we put by convention $\langle z^{*}, f(x)\rangle=+\infty,$ for any $x\notin\mathrm{dom}f.$ Thus, $z^{*}\circ f\;:\; X\longrightarrow\overline{\mathbb{R}}$ and $\mathrm{dom}(z^{*}\circ f)=\mathrm{dom}f$ and we say that $f$ is strict star $Z_+$-lower semicontinuous if $z^{*}\circ f$ is lower semicontinuous for all $z^{*}\in (Z_{+}^{*})^{\circ}.$ Moreover, we recall that the mapping $f$ is $Z_{+}$-convex if for every $\lambda\in[0,1]$ and $x_{1},x_{2}\in X$ we have
\[
 f(\lambda x_1 + (1-\lambda) x_2) \leq _{Z_+} \lambda  f(x_1) + (1-\lambda) f(x_2).
\]
Let $"\leq_{Y_{+}}"$ be the partial order on $Y$ induced a nonempty convex cone $Y_{+}\subset Y,$ we say that the mapping $g\;:\; Y\longrightarrow \overline{Z}$ is $(Y_{+},Z_{+})$-nondecreasing if for each $y_{1},y_{2}\in Y$ we have
\[y_{1}\leq_{Y_{+}} y_{2}\Longrightarrow g(y_{1})\leq_{Z_{+}} g(y_{2}).\]
Let  $h\;:\; X\longrightarrow \overline{Y}$ be a mapping,  then the composed mapping $g\circ h\;:\; X\longrightarrow \overline{Z}$ is defined by
\[
(g\circ h)(x):= \left\{
\begin{array}{ll}
 g(h(x)),\ \ \hbox{if} \ \ x \in {\mathrm{dom}\ h},\\
\\[1mm]
 +\infty_Z, \quad \; \hbox{otherwise.}
\end{array}
\right.
\]
It is easy to see that if $g\;:\; Y\longrightarrow \overline{Z}$ is $Z_{+}$-convex, $(Y_{+},Z_{+})$-nondecreasing and $h\;:\; X\longrightarrow \overline{Y}$ is $Y_{+}$-convex, then the composed mapping $g\circ h$ is $Z_{+}$-convex.

Now, we consider the following vector minimization problem
\[(VMP) \quad\min_{x\in C} f(x) \]
\noindent where $f\;:\; X\longrightarrow \overline{Z}$ is a mapping and $C$ is a nonempty subset of $X.$
\begin{definition}\label{der}
Let $\overline{x}$ be a feasible point of (VMP) i.e. $\overline{x}\in C.$  The point $\overline{x}$ is called properly efficient solution of the problem (VMP) in the sense of Henig if  there exists a convex cone $\ \hat{Z}_{+}\subset Z$  such that  $Z_{+}\setminus\{0\}\subset  \mathrm{int}\hat{Z}_{+}$ and
\[\nexists x\in C,\quad f(x)\lneqq_{\hat{Z}_{+}} f(\overline{x}).\]
\end{definition}
The set of Henig properly efficient solutions will be denoted by $E^{p}(f,C).$ The above notion of Henig properly efficient solution leads us to define the notion of Henig proper subdifferential of a vector valued mapping $f\;:\; X\longrightarrow \overline{Z}$ at $\overline{x}\in \mathrm{dom}f$ (see El Maghri-Laghdir\cite{maglag}). For $A\in L(X,Z)$
 \begin{multline*}
A\in \partial^{p} f(\bar x) \text{ if there exists}\;\hat{Z}_{+} \subsetneq Z\; \text{ convex cone such that } Z_{+} \setminus\{0\}\subseteq \mathrm{int} \hat{Z}_{+}, \,\\ \nexists x \in C,\  f( x)-f(\bar x)\lneqq_{\hat{Z}_+} A(x-\overline{x}),\;\forall x\in X
\end{multline*}
where $L(X,Z)$ denotes the space of linear continuous operators from X to Z.
This definition is justified by the importance of the following immediate property
\[\bar x \in E^{p}(f, X) \Longleftrightarrow 0 \in \partial^p
f(\bar x) .\]
For a given function $f\;:\; X\longrightarrow \overline{\mathbb{R}},$ the conjugate function $f^{*}\;:\; X^{*}\longrightarrow \overline{\mathbb{R}}\cup\{-\infty\}$ is defined by
\[f^{*}(x^{*}):=\underset{x\in X}\sup\{\langle x^{*},x\rangle-f(x)\}.\]
Recall that, for $\epsilon\geq 0,$ the $\epsilon$-subdifferential of $f$ at $\overline{x}\in \mathrm{dom}f$ is defined by
\[\partial_{\epsilon} f(\overline{x}):=\{x^{*}\in X^{*}\ :\  \langle x^{*},x-\overline{x}\rangle+f(\overline{x})-\epsilon\leq
f(x),\quad \forall x\in X\}.\]
If $\epsilon=0,$ the set $\partial f(\overline{x}):=\partial_{0}f(\overline{x})$ is the classical subdifferential of convex analysis, that is
\[\partial_{} f(\overline{x}):=\{x^{*}\in X^{*}\ :\  \langle x^{*},x-\overline{x}\rangle+f(\overline{x})\leq
f(x),\quad \forall x\in X\}.\]
Moreover, the Young-Fenchel equality holds
\[f^{*}(x^{*})+f(\overline{x})=\langle x^{*},\overline{x}\rangle\Longleftrightarrow x^{*}\in \partial f(\overline{x}).\]
As consequence of that
\[\left(x^{*},\langle x^{*},\overline{x}\rangle-f(\overline{x})\right)\in\mathrm{epi}f^{*}, \qquad  \forall x^{*}\in \partial f(\overline{x}).\]
For any subset $C\subset X,$ the vector indicator mapping $\delta^v_C : X \longrightarrow \overline{Z}$ of $C$ is given by
\[
\begin{array}{ll}
\delta^v_C(x)=
 \left\{
\begin{array}{ll}
 0 \qquad \quad \hbox{if } x \in C \\[1mm]
 +\infty_Z \quad \text{otherwise.}
\end{array}
\right.
\end{array}
\]
\noindent When $Z=\mathbb{R},$ the scalar indicator function is denoted by $\delta_C $. The vector indicator mapping appears to possess properties like the scalar one. For $\epsilon\geq 0,$ the $\epsilon$-normal at $\overline{x}\in C$ is defined by
\[N_{\epsilon}(\overline{x},C):=\{x^{*}\in X^{*}:\quad \langle x^{*},x-\overline{x}\rangle\leq\epsilon,\quad \forall x\in C\}.\]
If $\epsilon=0,$ $N(\overline{x},C):=N_{0}(\overline{x},C)$ is the usual normal cone at $\overline{x}$. Moreover, it is easy too see that the $\epsilon$-normal of $C$ at $\overline{x}$ is defined as the $\epsilon$-subdifferential of $\delta_{C}$ at $\overline{x}.$
\smallskip

The following theorem characterizes the epigraph of the conjugate of the sums of $m \ (m\geq2)$ functions (see  Bo\c{t}-Grad-Wanka \cite{botlivre}, Theorem 2.3.10, p. 37).
\begin{theorem}
  \label{epiofthesum}
 Let $f_{i}$ : $X\longrightarrow\overline{\mathbb{R}}$ be $m \ (m\geq2)$ proper, convex
and lower semicontinuous functions such that $\displaystyle \bigcap_{i=1}^{m}\mathrm{dom}f_{i}\neq\emptyset.$ Then  \[\mathrm{epi}(\sum_{i=1}^{m}f_{i})^{*}=\mathrm{cl}_{w^*\times \tau_{\mathbb{R}}}(\sum_{i=1}^{m}\mathrm{epi}f_{i}^{*})\]
where $cl_{w^*\times \tau_{\mathbb{R}}}$ denotes the weak closure on the product space $X^{*}\times \mathbb{R}.$
\end{theorem}
 The following theorem, show how  the epigraph of the conjugate function
can be expressed in terms of $\epsilon$-subdifferential (see Jeyakumar-Lee-Dinh \cite{jevv}).
\begin{theorem}\label{jekumyaar}
Let $f$ : $X\longrightarrow\overline{\mathbb{R}}$ be a proper convex lower
 semicontinuous  function and let $\overline{x}\in \mathrm{dom}f.$ Then
\[\mathrm{epi}{f}^{*}=\bigcup_{\epsilon \geq0}\{(x^{*},\langle x^{*},\overline{x}\rangle+\epsilon-f(\overline{x}))\; :\;
x^{*}\in\partial_{\epsilon}f(\overline{x})\}.\]
\end{theorem}
Let us recall a Br{\o}ndsted-Rockafellar theorem (see Thibault \cite{thibault}).
\begin{theorem}
\label{thibault} Let $X$ be a Banach space and $f$ : $X\longrightarrow\overline{{\mathbb{R}}}$  be a proper, convex and lower semicontinuous function.\ Then for any real $\epsilon>0$ and $\overline{x}^{*}\in\partial_{\epsilon}f(\overline{x}),$ there exist $x\in\mathrm{dom}f$ and $x_{}^{\ast }\in \partial f(x_{})$ such that
\begin{enumerate}
  \item $\left\Vert x_{ }-\overline{x}\right\Vert \leq \sqrt{\epsilon },$
  \item $\left\Vert x_{}^{\ast }-\overline{x}^{\ast }\right\Vert \leq \sqrt{\epsilon},$
  \item $\left\vert f(x_{})-f(\overline{x})-\left\langle x_{}^{\ast
},x_{}-\overline{x}\right\rangle \right\vert \leq 2\epsilon.$
\end{enumerate}
\end{theorem}
 We conclude this section by recalling the following important scalarization theorem (see El Maghri-Laghdir \cite{maglag}).

 \begin{theorem}\label{scalarization} Let $X, \ Z$  be two Hausdorff topological vector spaces and  $ f: X \longrightarrow \overline{Z}$ be $ Z_+$-convex vector valued mapping. For $\bar{x} \in X$, we have
\[\partial^{p} f(\bar x) = \bigcup_{z^* \in (Z_+^{*})^{\circ}} \{ A \in L(X,Z) : z^*\circ A \in \partial (z^*\circ f)(\bar x)\}.\]

\end{theorem}

\section{Sequential Pareto proper subdifferential of the sums of convex vector valued mappings with a composition of two convex vector valued mappings}
 In what follows $(X,\|.\|_X)$ and $(Y,\|.\|_Y)$ stand for two real reflexive
Banach spaces,  $(Z,\|.\|_Z)$ be a real normed vector space and
$(X^*,\|.\|_X*)$, $(Y^*,\|.\|_Y*), $  $(Z^*,\|.\|_Z*)$ their respective topological dual spaces. On $X\times Y$ we use the norm $\|(x, y)\| = \sqrt{\|x\|^2
+ \|y\|^2}$ for any $(x, y) \in X\times Y.$ Similarly, we define the norm
on $X^*\times Y^*$. Further, let $(x_{n})_{n\in\mathbb{N}}$ be a sequence in $X$ (resp. $(x^{*}_{n})_{n\in\mathbb{N}}$ be a sequence in $X^{*}$) and $x \in  X$ (resp. $x^{*}\in X^{*}$), we write $x_{n}\xrightarrow{\parallel.\parallel_{X}}x$ (resp. $x^{*}_{n}\xrightarrow{\parallel.\parallel_{X^{*}}}x^{*}$) if $\|x_{n}-x\|_{X}\longrightarrow 0$ (resp. $\|x^{*}_{n}-x^{*}\|_{X^{*}}\longrightarrow 0$) as $n\longmapsto+\infty.$\\
Our aim in this section is to formulate in the absence of constraint qualifications, a formula for the Henig proper subdifferential of the convex mapping $\partial^{p}(\displaystyle\sum_{i=1}^{m}f_{i}+g\circ h)(\overline{x}), $ where $f_{i}:X \longrightarrow \overline{Z}$ $(i=1,\hdots,m)$ be
a proper and $Z_+$-convex mappings, $h:X \longrightarrow \overline{Y} $ be a proper and $Y_+$-convex mapping, and  $g:Y
\longrightarrow \overline{Z}$ be a proper, $Z_+$-convex
and $(Z_{+},Y_{+})$-nondecreasing mapping. Let us consider the following auxiliary mappings
\[
\qquad\qquad\qquad\begin{array}{lrcl}
F_{i}\ : & X\times Y& \longrightarrow&\overline{Z}\\
&(x,y)& \longrightarrow& F_{i}(x,y):=f_{i}(x) ,
\end{array}\quad(i=1,\hdots,m)
\]
\[
\begin{array}{lrcl}
G\ : & X\times Y& \longrightarrow&\overline{Z}\\
&(x,y)& \longrightarrow& G(x,y):=g(y),
\end{array}
\]
\[
\qquad\begin{array}{lrcl}
H\ : & X\times Y& \longrightarrow&\overline{Z}\\
&(x,y)& \longrightarrow& H(x,y):=\delta^{v}_{\mathrm{epi}h}(x,y) .
\end{array}
\]
\begin{lemma}\label{lemma1} Let $z^{*}\in (Z_{+}^{*})^{\circ}$ and $(x^{*},y^{*},s)\in X^{*}\times Y^{*}\times \mathbb{R}.$ Then, we have
\begin{enumerate}
  \item $(x^{*},y^{*},s)\in\mathrm{epi}(z^{*}\circ F_{i})^{*}\Longleftrightarrow (x^{*},s)\in \mathrm{epi}(z^{*}\circ f_{i})^{*}\;\text{and}\; y^{*}=0, \quad (i=1,\hdots,m).$
 \item $(x^{*},y^{*},s)\in\mathrm{epi}(z^{*}\circ G)^{*}\Longleftrightarrow x^{*}=0\;\text{and}\;(y^{*},s)\in \mathrm{epi}(z^{*}\circ g)^{*}.$
 \item  $(x^{*},y^{*},s)\in\mathrm{epi}(z^{*}\circ H)^{*}\Longleftrightarrow\;(x^{*},s)\in \mathrm{epi}(-y^{*}\circ h)^{*}
     \text{and}\; -y^{*}\in Y_{+}^{*}.$
  \end{enumerate}
\end{lemma}

\begin{proof}

It is easy to see that the conjugate functions associated to the functions  $z^{*}\circ F_{i}\ (i=1,\hdots,m),$ $z^{*}\circ G$ and $z^{*}\circ H$ are given for any $(x^{*},y^{*})\in X^{*}\times Y^{*},$ by
\begin{eqnarray}
  (z^{*}\circ F_{i})^{*}(x^{*},y^{*})&=&(z^{*}\circ f_{i})^{*}(x^{*})+\delta_{\{0\}}(y^{*}), \quad(i=1,\hdots,m)\label{1}\\
  (z^{*}\circ G)^{*}(x^{*},y^{*})&=&(z^{*}\circ g)^{*}(y^{*})+\delta_{\{0\}}(x^{*})\label{2}\\
  (z^{*}\circ H)^{*}(x^{*},y^{*})&=&(-y^{*}\circ h)^{*}(x^{*})+\delta_{Y_{+}^{*}}(-y^{*})\label{3}.
\end{eqnarray}

\noindent i) Let $(x^{*},y^{*},s)\in\mathrm{epi}(z^{*}\circ F_{i})^{*},$ $(i=1,\hdots,m),$ then we have
\[(z^{*}\circ F_{i})^{*}(x^{*},y^{*})\leq s\]
and by (\ref{1}) we get  \[(z^{*}\circ f_{i})^{*}(x^{*})+\delta_{\{0\}}(y^{*})\leq s,\quad (i=1,\hdots,m),\]
i.e. \[(x^{*},s)\in \mathrm{epi}(z^{*}\circ f_{i})^{*}\;\text{and}\; y^{*}=0, \quad (i=1,\hdots,m).\]
By applying the same arguments as above we obtain easily ii) and iii).
\end{proof}

Now, we state the sequential Henig proper subdifferential of the convex mapping $\displaystyle\sum_{i=1}^{m}f_{i}+g\circ h$ by means of the epigraphs of the conjugate of data
vector valued mappings.
\begin{theorem}\label{theorem1}
Let $f_{1},\hdots,f_{m}$ : $X\longrightarrow \overline{Z}$ be $m \ (m\geq 2)$
proper, $Z_{+}$-convex and strict star $Z_+$-lower semicontinuous mappings, $g$ : $Y\longrightarrow \overline{Z}$ be proper, $Z_{+}$-convex,
strict star $Z_+$-lower semicontinuous and $(Y_{+},Z_{+})$-nondecreasing mapping and $h$ : $X\longrightarrow \overline{Y}$ be proper,
$Y_{+}$-convex and $Y_{+}$-epi-closed mapping. Let $\overline{x}\in(\displaystyle
\bigcap_{i=1}^{m}\mathrm{dom}f_{i})\cap \mathrm{dom}h\cap h^{-1}(\mathrm{dom} g).$ Then, $A\in\partial^{p}\left(\displaystyle \sum_{i=1}^{m}f_{i}+g\circ h\right)(\overline{x})$ if and only if, there exist $z^{*}\in (Z_{+}^{*})^{\circ},$ $(x_{i,n}^{*},r_{i,n})\in\mathrm{epi}(z^{*}\circ f_{i})^{*},$ $(i=1,\hdots,m),$ $(y_{n}^{*},s_{n})\in\mathrm{epi}(z^{*}\circ g)^{*},$ $v_{n}^{*}\in -Y_{+}^{*}$ and $(u_{n}^{*},t_{n})\in\mathrm{epi}(-v_{n}^{*}\circ h)^{*}$ such that
\begin{numcases}{}
\displaystyle\sum_{i=1}^{m} x^{*}_{i,n}+u_{n}^{*}\underset{n\longmapsto+\infty}{\xrightarrow{\parallel.\parallel_{X^{*}}}}z^{*}\circ A\nonumber\\
y_{n}^{*}+v_{n}^{*}\underset{n\longmapsto+\infty}{\xrightarrow{\parallel.\parallel_{Y^{*}}}}0\nonumber\\
\displaystyle\sum_{i=1}^{m}r_{i,n}+s_{n}+t_{n}\underset{n\longmapsto+\infty}{\xrightarrow{\qquad}}(z^{*}\circ A)(\overline{x})-\displaystyle\sum_{i=1}^{m}(z^{*}\circ f_{i})(\overline{x})-(z^{*}\circ g)(h(\overline{x})).\nonumber
\end{numcases}
\end{theorem}

\begin{proof}
Let $A\in\partial^{p}\left(\displaystyle \sum_{i=1}^{m}f_{i}+g\circ h\right)(\overline{x}).$ According to scalarization Theorem \ref{scalarization}, there exists $z^{*}\in (Z_{+}^{*})^{\circ}$ such that
\begin{equation}\label{Eq1}
  z^{*}\circ A\in\partial\left(\displaystyle \sum_{i=1}^{m}z^{*}\circ f_{i}+z^{*}\circ g\circ h\right)(\overline{x}).
\end{equation}
By introducing the scalar indicator function $\delta^{}_{\mathrm{epi}h}$ and by adopting the convention $z^{*}(+\infty_{Z})=+\infty,$ it easy to check that $z^{*}\circ\delta^{v}_{\mathrm{epi}h}=\delta^{}_{\mathrm{epi}h}$ and  by using the monotonicity of the mapping $g$ it follows that (\ref{Eq1}) becomes equivalent to
\[(z^{*}\circ A,0)\in\partial\left(\displaystyle \sum_{i=1}^{m}z^{*}\circ F_{i}+ z^{*}\circ G+z^{*}\circ H\right)(\overline{x},h(\overline{x}))\]
i.e.
\begin{multline*}\left(\displaystyle \sum_{i=1}^{m}z^{*}\circ F_{i}+ z^{*}\circ G+z^{*}\circ H\right)^{*}(z^{*}\circ A,0)+\left(\displaystyle \sum_{i=1}^{m}z^{*}\circ F_{i}+ z^{*}\circ G+z^{*}\circ H\right)(\overline{x},h(\overline{x}))\\=\langle(z^{*}\circ A,0),(\overline{x},h(\overline{x}))\rangle=\langle z^{*}\circ A ,\overline{x}\rangle
\end{multline*}
and hence we get
\begin{multline}\label{4}
\left((z^{*}\circ A,0),\langle z^{*}\circ A,\overline{x}\rangle-\left(\displaystyle \sum_{i=1}^{m}z^{*}\circ F_{i}+ z^{*}\circ G+z^{*}\circ H\right)(\overline{x},h(\overline{x}))\right)\\\in\mathrm{epi}\left(\displaystyle \sum_{i=1}^{m}z^{*}\circ F_{i}+ z^{*}\circ G+z^{*}\circ H\right)^{*}.
\end{multline}
It is easy to see that the mappings $ F_{i},\;(i=1,\hdots,m),$ $G$ and $H$ are proper, $Z_{+}$-convex and strict star $Z_{+}$-lower semicontinuous on $X\times Y$ and as $z^{*}$ is $Z_{+}$-nondecreasing, it follows that the scalar functions $z^{*}\circ F_{i},\; (i=1,\hdots,m),$ $z^{*}\circ G$ and $z^{*}\circ H$ are proper, convex and lower semicontinuous. Let us note that $\mathrm{dom}(z^{*}\circ F_{i})=\mathrm{dom} f_{i}\times Y,\; (i=1,\hdots,m),$
$\mathrm{dom}(z^{*}\circ G)= X\times\mathrm{dom} g$ and $\mathrm{dom}(z^{*}\circ H)=\mathrm{epi}h$ and the condition $\overline{x}\in(\displaystyle
\bigcap_{i=1}^{m}\mathrm{dom}f_{i})\cap \mathrm{dom}h\cap h^{-1}(\mathrm{dom} g)$ can be written equivalently as $(\overline{x},h(\overline{x}))\in (\displaystyle \bigcap_{i=1}^{m}\mathrm{dom}F_{i})\cap \mathrm{dom}G\cap \mathrm{dom}H.$ Thus, the functions $z^{*}\circ F_{i}\; (i=1,\hdots,m),$ $z^{*}\circ G$ and $z^{*}\circ H,$ satisfy together all the  assumptions of Theorem \ref{epiofthesum} and hence it follows from (\ref{4}) that
\begin{multline*}
\left((z^{*}\circ A,0),\langle z^{*}\circ A,\overline{x}\rangle-\left(\displaystyle \sum_{i=1}^{m}z^{*}\circ F_{i}+ z^{*}\circ G+z^{*}\circ H\right)(\overline{x},h(\overline{x}))\right)\\\in \mathrm{cl}_{w^{*}\times\tau_{\mathbb{R}}}\left(\displaystyle \sum_{i=1}^{m}\mathrm{epi}(z^{*}\circ F_{i})^{*}+ \mathrm{epi}(z^{*}\circ G)^{*}+\mathrm{epi}(z^{*}\circ H)^{*}\right)\\= \mathrm{cl}_{\|.\|_{X^{*}\times Y^{*}\times\tau_{\mathbb{R}}}}\left(\displaystyle \sum_{i=1}^{m}\mathrm{epi}(z^{*}\circ F_{i})^{*}+ \mathrm{epi}(z^{*}\circ G)^{*}+\mathrm{epi}(z^{*}\circ H)^{*}\right)
\end{multline*}
and therefore there exist $\left((x^{*}_{i,n},y^{*}_{i,n}),r_{i,n}\right),\; (i=1,\hdots,m),$ $\left((x^{*}_{n},y^{*}_{n}),s_{n}\right)$ and $\left((u^{*}_{n},v^{*}_{n}),t_{n}\right)\in X^{*}\times Y^{*}\times\mathbb{R},$ satisfying
\begin{equation}
\left.
    \begin{array}{ll}
    \left((x^{*}_{i,n},y^{*}_{i,n}),r_{i,n}\right)\in\mathrm{epi}(z^{*}\circ F_{i})^{*},\; (i=1,\hdots,m)\\
    \\[0.25mm]\left((x^{*}_{n},y^{*}_{n}),s_{n}\right)\in\mathrm{epi}(z^{*}\circ G)^{*} \\
     \\[0.25mm]
     \noindent\left((u^{*}_{n},v^{*}_{n}),t_{n}\right)\in\mathrm{epi}(z^{*}\circ H)^{*}
\end{array}
\right\}\label{6}
\end{equation}
\noindent such that
\begin{multline}\label{5}
\displaystyle \sum_{i=1}^{m}\left((x^{*}_{i,n},y^{*}_{i,n}),r_{i,n}\right)
+\left((x^{*}_{n},y^{*}_{n}),s_{n}\right)+\left((u^{*}_{n},v^{*}_{n}),t_{n}\right)\\
\xrightarrow{{\parallel.\parallel_{X^{*}\times Y^{*}}}}\left((z^{*}\circ A,0),\langle z^{*}\circ A,\overline{x}\rangle-\left(\displaystyle \sum_{i=1}^{m}z^{*}\circ F_{i}+ z^{*}\circ G+z^{*}\circ H\right)(\overline{x},h(\overline{x}))\right).
\end{multline}
By applying Lemma \ref{lemma1}, (\ref{6}) may be rewritten as
\begin{numcases}{}
\left((x^{*}_{i,n},y^{*}_{i,n}),r_{i,n}\right)\in\mathrm{epi}(z^{*}\circ F_{i})^{*}\Longleftrightarrow (x^{*}_{i,n},r_{i,n})\in \mathrm{epi}(z^{*}\circ f_{i})^{*}\;\text{and}\; y^{*}_{i,n}=0,\nonumber\\\nonumber
\\[0.25mm]
\left((x^{*}_{n},y^{*}_{n}),s_{n}\right)\in\mathrm{epi}(z^{*}\circ G)^{*}\Longleftrightarrow x^{*}_{n}=0\;\text{and}\;(y^{*}_{n},s_{n})\in \mathrm{epi}(z^{*}\circ g)^{*}\nonumber\\ \nonumber
\\[0.25mm]
\left((u^{*}_{n},v^{*}_{n}),t_{n}\right)\in\mathrm{epi}(z^{*}\circ H)^{*}\Longleftrightarrow (u^{*}_{n},t_{n})\in \mathrm{epi}(-v_{n}^{*}\circ h)^{*}\;\text{and}\; -v_{n}^{*}\in Y_{+}^{*}.\nonumber
\end{numcases}

Since  $x^{*}_{n}=0,$ $y^{*}_{i,n}=0,$ $(i=1,\hdots,m)$ and  $(\overline{x},h(\overline{x}))\in \mathrm{epi}h,$ then the expression  (\ref{5}) becomes equivalent to

\begin{numcases}{}
\displaystyle \sum_{i=1}^{m}x^{*}_{i,n}+u^{*}_{n} \xrightarrow{{\parallel.\parallel_{X^{*}}}}z^{*}\circ A\nonumber\\
y^{*}_{n}+v^{*}_{n}\xrightarrow{{\parallel.\parallel_{ Y^{*}}}}0\nonumber\\
\displaystyle \sum_{i=1}^{m}r_{i,n}+s_{n}+t_{n}\longrightarrow\langle z^{*}\circ A,\overline{x}\rangle-\displaystyle \sum_{i=1}^{m}z^{*}\circ f_{i}(\overline{x})- z^{*}\circ g(h(\overline{x})).\nonumber
\end{numcases}
The proof is complete.
\end{proof}
\section{Sequential proper efficiency optimality conditions}
In this section, we are concerned with the general multiobjective fractional programming problem
\begin{equation*}
\mathrm{(P)}\; \inf_{\substack{{x}\in C\\h(x)\in -Y_{+}}}\left\{\left(\frac{f_{1}(x)}{g_{1}(x)},\hdots,\frac{f_{m}(x)}{g_{m}(x)}\right)\right\},
\end{equation*}

\noindent with $C\subset X$ be a nonempty, closed and convex subset and $Y_{+}\subset Y$ be a nonempty closed convex cone. The functions $f_{i},$ {-}$g_{i}$ : $X\longrightarrow\mathbb{R},$ $ (i=1,\hdots,m)$ are convex, lower semicontinuous and $h$ : $X\longrightarrow \overline{Y}$ is a proper, $Y_{+}$-convex and $Y_{+}$-epi-closed mapping. We assume that $f_{i}(x)\geq 0,\;g_{i}(x)>0 $ $\ (i=1,\hdots,m).$  Let $e_i$ denote the ith unit coordinate vector and $e$ the vector of ones in $\mathbb{R}^m$ . For $\epsilon \geq 0$, the positive  hull of the subset $S:=\{e_i + \epsilon e : i = 1,\cdots, m\}$ is defined by
\[K_{\epsilon}:= \{\displaystyle \sum_{i=1}^{m}\alpha_i(e_i + \epsilon e) : \; \alpha_i \geq 0 \;  \}.\]
In fact $K_{\epsilon}$ is a convex cone and the origin belongs to $K_{\epsilon}$. The positive polar cone of $K_{\epsilon}$ is denoted by
\[ K_{\epsilon}^*:=\{v\in \mathbb{R}^m: \; \langle v, y\rangle \geqq 0, \quad \forall y\in K_{\epsilon}\}.\]
It is easy to see that
\[K_{\epsilon} \backslash\{0\}\subset  \mathrm{int}(\mathbb{R}_+^m)\subset\mathbb{R}_+^m\backslash\{0\}\subset \mathrm{int}(K_{\epsilon}^*).\]
We endow the finite-dimensional space $Z:=\mathbb{R}^{m}$ with
its natural order induced by the nonnegative orthant $Z_{+}:=\mathbb{R}^{m}_{+},$ and we shall use the following characterization of proper efficiency (see Luc-Soleimani-damaneh \cite{luc}).
\begin{proposition}\label{l60} A point
$\overline{x}\in C\cap h^{-1}(-Y_{+})$ is properly efficient solution of the problem $(P)$ if and only if there exists some $\epsilon>0$ and there is no $x\in C\cap h^{-1}(-Y_{+})$ such that
\[\left(\frac{f_{1}(x)}{g_{1}(x)}-\frac{f_{1}(\overline{x})}{g_{1}(\overline{x})},\hdots,\frac{f_{m}(x)}{g_{m}(x)}-\frac{f_{m}(\overline{x})}{g_{m}(\overline{x})}\right)\in -K_{\epsilon}^{*},\]
\end{proposition}
We associate to problem (P) the multiobjective convex minimization problem
\[({P}_{\overline{x}})\quad \inf_{x\in C\cap h^{-1}(-Y_{+})}\left\{\left(f_{1}(x)-v_{1}g_{1}({x}),\hdots,f_{m}(x)-v_{m}g_{m}({x})\right)\right\},\]
\noindent where $\overline{x}\in C\cap h^{-1}(-Y_{+})$ and $\nu_{i}:=\frac{f_{i}(\overline{x})}{g_{i}(\overline{x})} \; (i=1,\hdots,m)$. The problem (${P}_{\overline{x}}$) is intimately related to $(P).$ The crucial relationship between $(P)$ and (${P}_{\overline{x}}$), which will serve our purposes, is stated in the following lemma
\begin{lemma}\label{l61}
A point $\overline{x}\in C\cap h^{-1}(-Y_{+})$ is  Henig properly efficient solution for problem (P) if and only if, $\overline{x}$ is  Henig properly efficient solution for problem $(P_{\overline{{x}}}).$
\end{lemma}
\begin{proof} ($\Longrightarrow$). Suppose in the contrary that $\overline{x}$ is not Henig proper efficient solution for $(P_{\overline{x}})$, then  it follows from Proposition \ref{l60} that for any $\epsilon>0,$ there exists some $x_{0}\in C\cap h^{-1}(-Y_{+})$ such that

\[\left(f_{1}(x_{0})-\nu_{1}g_{1}(x_{0}),\hdots,f_{i}(x_{0})-\nu_{i}g_{i}(x_{0})\right)\in-K_{\epsilon}^{*},\]
i.e.
\begin{equation}\label{dn1}
\displaystyle\sum_{i=1}^{m}\left(f_{i}(x_{0})-\nu_{i}g_{i}(x_{0})\right)\left(\alpha_{i}+\epsilon\displaystyle\sum_{j=1}^{m}\alpha_{j}\right)\leq 0,\quad\forall (\alpha_{1},\hdots,\alpha_{m})\in\mathbb{R}_{+}^{m}.
\end{equation}
By using the fact that $g_{i}(x_{0})>0\; (i=1,\hdots,m)$ and
by substituting respectively  in (\ref{dn1}) $\alpha_{i}$ and  $\alpha_{j}$ by $\frac{\alpha_{i}}{g_{i}(x_{0})}$ and $\frac{\alpha_{j}}{g_{i}(x_{0})}$ $(j=1,\hdots,m)$ we obtain that for any $\epsilon>0,$ we have
 \[\displaystyle\sum_{i=1}^{m}\left(\frac{f_{i}(x_{0})}{g_{i}(x_{0})}-\nu_{i}\right)\left(\alpha_{i}+\epsilon\displaystyle\sum_{j=1}^{m}\alpha_{j}\right)\leq 0,\;\forall (\alpha_{1},\hdots,\alpha_{m})\in\mathbb{R}_{+}^{m},\]
which means for any $\epsilon>0$ \begin{equation*}
\left(\frac{f_{1}(x_{0})}{g_{1}(x_{0})}-\nu_{1},\hdots,\frac{f_{m}(x_{0})}{g_{m}(x_{0})}-\nu_{m}\right)\in -K_{\epsilon}^{*}.
\end{equation*}
According to Proposition \ref{l60} this contradicts the fact that $\overline{x}$ is Henig proper efficient solution for the problem  $(P).$

$(\Longleftarrow)$.  We proceed by contradiction. Assume that $\overline{x}$ is not Henig proper efficient solution of the problem (${P}$), then according to Proposition \ref{l60}, we have  for any $\epsilon>0,$ there exists some  $x_{0}\in C\cap h^{-1}(-Y_{+})$ such that
\begin{equation*}
\left(\frac{f_{1}(x_{0})}{g_{1}(x_{0})}-\nu_{1},\hdots,\frac{f_{m}(x_{0})}{g_{m}(x_{0})}-\nu_{m}\right)\in -K_{\epsilon}^{*}.
\end{equation*}
So, for any $(\alpha_{1},\hdots,\alpha_{m})\in\mathbb{R}_{+}^{m},$ we have
\begin{equation}\label{lc1}
  \displaystyle\sum_{i=1}^{m}\left(\frac{f_{i}(x_{0})-\nu_{i}g_{i}(x_{0})}{g_{i}(x_{0})}\right)\left(\alpha_{i}+\epsilon\displaystyle\sum_{j=1}^{m}\alpha_{j}\right)\leq 0.\end{equation}
 Since $g_{i}(x_{0})>0\; (i=1,\hdots,m),$ then by substituting respectively in (\ref{lc1}) $\alpha_{i}$ and  $\alpha_{j}$ by $\alpha_{i}g_{i}(x_{0})$ and  $\alpha_{j}g_{i}(x_{0})$ $(j=1,\hdots,m)$  we get
 that for any $\epsilon>0,$
\[\displaystyle\sum_{i=1}^{m}\left(f_{i}(x_{0})-\nu_{i}g_{i}(x_{0})\right)\left(\alpha_{i}+\epsilon\displaystyle\sum_{j=1}^{m}\alpha_{j}\right)\leq 0,\quad\forall (\alpha_{1},\hdots,\alpha_{m})\in\mathbb{R}_{+}^{m},\]
which yields that $\epsilon>0,$
\[\left(f_{1}(x_{0})-\nu_{1}g_{1}(x_{0}),\hdots,f_{i}(x_{0})-\nu_{i}g_{i}(x_{0})\right)\in-K_{\epsilon}^{*},\]
this contradicts, by virtue of Proposition \ref{l60} the fact that $\overline{x}$ is Henig proper efficient solution of the problem $(P_{\overline{x}}).$

\end{proof}

Now, by using this equivalence result in conjunction with Theorem \ref{theorem1}, we can establish the first characterization of sequential optimality conditions for problem (P) by means of the epigraphs of the conjugate of data functions.
\begin{theorem}\label{4.2}
Let $\overline{x}\in C\cap h^{-1}(-Y_{+})$ and $\nu_{i}:=\frac{f_{i}(\overline{x})}{g_{i}(\overline{x})} \; (i=1,\hdots,m).$ Then, $\overline{x}$ is a properly efficient solution of the problem (P) in the sense of Henig, if and only if, there exist $(\lambda_{1},\hdots,\lambda_{m})\in(\mathbb{R}_{+}\setminus\{0\})^{m},$ $(x^{*}_{i,n},a_{i,n})\in\mathrm{epi}(\lambda_{i}f_{i})^{*},$ $(w^{*}_{i,n},b_{i,n})\in\mathrm{epi}(\lambda_{i}\nu_{i}(-g_{i}))^{*},$ $(c^{*}_{n},d_{n})\in\mathrm{epi}\delta_{C}^{*}, \; y^{*}_{n}\in Y_+^*, \; s_{n}\in\mathbb{R_+}, \; v_{n}^{*}\in-Y_{+}^{*}$ and $(u^{*}_{n},t_{n})\in\mathrm{epi}(-v_{n}^{*}\circ h)^{*}$ such that
\begin{numcases}{}
\displaystyle\sum_{i=1}^{m}x_{i,n}^{*}+\displaystyle\sum_{i=1}^{m}w_{i,n}^{*}+c_{n}^{*}+u_{n}^{*}\xrightarrow{{\parallel.\parallel_{X^{*}}}}0\nonumber\\
y_{n}^{*}+v_{n}^{*}\xrightarrow{{\parallel.\parallel_{Y^{*}}}}0\nonumber\\
\displaystyle\sum_{i=1}^{m}a_{i,n}+\displaystyle\sum_{i=1}^{m}b_{i,n}+d_{n}+t_{n}+s_{n}\xrightarrow{\qquad}0.\nonumber
\end{numcases}
\end{theorem}

\begin{proof}
According to Lemma \ref{l61}, $\overline{x}$ is Henig properly efficient solution for problem (P) if and only if, $\overline{x}$ is Henig properly efficient solution for problem $(P_{\overline{x}}).$ By introducing the vector indicator mappings $\delta_{C}^{v}$ and $\delta^{v}_{-Y_{+}},$ the problem $(P_{\overline{x}})$ may be written equivalently as
\[\inf_{\substack{{x}\in X}}\left\{F_{\overline{x}}(x)+\delta^{v}_{C}(x)+(\delta^{v}_{-Y_{+}}\circ h)(x)\right\}\]

\noindent where $F_{\overline{x}}\ :\ X\longrightarrow\mathbb{R}^{m}$ is defined for any $x\in X,$ by
 \[F_{\overline{x}}(x):=\left(f_{1}(x)-\nu_{1}g_{1}(x),\hdots,f_{m}(x)-\nu_{m}g_{m}(x)\right).\]
Hence $\overline{x}$ is  Henig properly efficient solution for problem (P) if and only if,
\begin{equation}\label{Eq2}
0\in\partial^{p}\left(F_{\overline{x}}+\delta^{v}_{C}+\delta^{v}_{-Y_{+}}\circ h\right)(\overline{x}).
\end{equation}
Let us consider the following vector mappings $L_{i}$ : $X\longrightarrow\overline{\mathbb{R}^{m}},$ $(i=1,\hdots,2m+1)$ defined by
\begin{empheq}[left={L_{i}(x):=}\empheqlbrace]{align*}
 & \left(0,\hdots,f_{i}(x),\hdots,0\right)&\;\mathrm{if}\; &(i=1,\hdots,m)\\
& \left(0,\hdots,\nu_{m-i}(-g_{m-i}(x)),\hdots,0\right) &\;\mathrm{if}\; &(i=m,\hdots,2m)\\
& \delta^{v}_{C}(x)&\;\mathrm{if}\; &i=2m+1,
\end{empheq}
 where the effective domain of the mappings $L_{i},$ $(i=1,\hdots,2m+1)$ are given by
\begin{empheq}[left={\mathrm{dom}L_{i}:=}\empheqlbrace]{align*}
& \mathrm{dom}f_{i}=X&\;\mathrm{if}\; &(i=1,\hdots,m)\\
& \mathrm{dom}[\nu_{i-m}(-g_{i-m})]=X&\;\mathrm{if}\; &(i=m+1,\hdots,2m)\\
& \mathrm{dom}\delta^{v}_{C}=C &\;\mathrm{if}\; &i=2m+1.
\end{empheq}
It is easy to see that the mappings $L_{i}$ : $X\longrightarrow\overline{\mathbb{R}^{m}},$ $(i=1,\hdots,2m+1)$ are proper, $\mathbb{R}_{+}^{m}$-convex and strict star $\mathbb{R}_{+}^{m}$-lower semicontinuous. By means of these notations the expression (\ref{Eq2}) may be written equivalently as
\[0\in\partial^{p}\left(\displaystyle\sum_{i=1}^{2m+1}L_{i}+\delta^{v}_{-Y_{+}}\circ h\right)(\overline{x}).\]
Let us note that the mapping $\delta_{-Y_{+}}^{v}$ is proper, $\mathbb{R}_{+}^{m}$-convex and strict star $\mathbb{R}_{+}^{m}$-lower semicontinuous since $Y_{+}$ is  a nonempty convex and closed cone. Moreover, let us recall that $\delta_{-Y_{+}}^{v}$ is $(Y_{+},\mathbb{R}_{+}^{m})$-nondecreasing (see
\cite{maglag}) and the condition $\bar{x}\in C\cap h^{-1}(-Y_{+}),$ can be equivalently rewritten as $\bar{x}\in(\displaystyle
\bigcap_{i=1}^{2m+1}\mathrm{dom}L_{i})\cap \mathrm{dom}h\cap h^{-1}(\mathrm{dom} \delta_{-Y_{+}}^{v}).$ Hence, the mapping $L_{i}$ $(i=1,\hdots,2m+1),$ $\delta_{-Y_{+}}^{v}$ and $h$ satisfy together all the assumptions of Theorem \ref{theorem1} and then there exist $z^{*}=(\lambda_{1},\hdots,\lambda_{m})\in((\mathbb{R}_{+}^{m})^{*})^{\circ}=(\mathbb{R}_{+}\setminus\{0\})^{m},$
$(\overline{x}_{i,n}^{*},r_{i,n})\in \mathrm{epi}(z^{*}\circ L_{i})^{*}\;(i=1,\hdots,2m+1),$ $(y_{n}^{*},s_{n})\in \mathrm{epi}(z^{*}\circ \delta_{-Y_{+}}^{v})^{*},$ $v_{n}^{*}\in -Y_{+}^{*}$ and $(u_{n}^{*},t_{n})\in \mathrm{epi}(-v_{n}^{*}\circ h)^{*}$ such that
\begin{numcases}{}
\displaystyle\sum_{i=1}^{2m+1}\overline{x}_{i,n}^{*}+u_{n}^{*}\xrightarrow{{\parallel.\parallel_{X^{*}}}}0\label{r1}\\
y_{n}^{*}+v_{n}^{*}\xrightarrow{{\parallel.\parallel_{Y^{*}}}}0\nonumber\\
\displaystyle\sum_{i=1}^{2m+1}r_{i,n}+t_{n}+s_{n}\xrightarrow{\qquad}-\displaystyle\sum_{i=1}^{2m+1}(z^{*}\circ L_{i})(\overline{x})-(z^{*}\circ\delta_{-Y_{+}}^{v})(h(\overline{x})).\label{7}
\end{numcases}
It is easy to check that $z^{*}\circ \delta_{C}^{v}=\delta_{C}$ and $z^{*}\circ
\delta_{-Y_{+}}^{v}=\delta_{-Y_{+}}.$ Therefore, for each $i=1,\hdots,2m+1$ the conditions $(\overline{x}_{i,n}^{*},r_{i,n})\in \mathrm{epi}(z^{*}\circ L_{i})^{*},$ $(y_{n}^{*},s_{n})\in \mathrm{epi}(z^{*}\circ \delta_{-Y_{+}}^{v})^{*}$ and (\ref{7}) can be rewritten by means of data functions $f_{i},$ $g_{i},$ $\delta_{C}$ and $\delta_{-Y_{+}}$ as follows
\[{(\overline{x}_{i,n}^{*},r_{i,n})\in \mathrm{epi}(z^{*}\circ L_{i})^{*}\Longleftrightarrow}\left\{
    \begin{array}{ll}
         (x_{i,n}^{*},a_{i,n}):=(\overline{x}_{i,n}^{*},r_{i,n})\in \mathrm{epi}(\lambda_{i} f_{i})^{*},\\
          \qquad\quad\qquad\qquad\quad\qquad\qquad\quad\qquad\mbox{if }\; i\in\{1,\hdots,m\} \\
         (w_{i,n}^{*},b_{i,n}):=(\overline{x}_{i+m,n}^{*},r_{i+m,n})\in \mathrm{epi}(\lambda_{i}\nu_{i} (-g_{i}))^{*},\\
         \qquad\quad\qquad\quad\qquad\quad\qquad\quad\quad\quad \mbox{if}\; i\in\{1,\hdots,m\}\\
         (c_{n}^{*},d_{n}):=(\overline{x}_{2m+1,n}^{*},r_{2m+1,n})\in \mathrm{epi}(z^{*}\circ \delta_{C}^{v})^{*}=\\  \qquad\quad\qquad\quad\qquad\quad\mathrm{epi}(\delta_{C})^{*},\mbox{if}\;i=2m+1,
    \end{array}
\right.
\]

\[(y_{n}^{*},s_{n})\in \mathrm{epi}(z^{*}\circ \delta_{-Y_{+}}^{v})^{*}=\mathrm{epi}(\delta_{-Y_{+}})^{*}=\mathrm{epi}(\delta_{Y_{+}^{*}})=Y_{+}^{*}\times\mathbb{R}_{+}\]
\noindent and
\begin{multline*}
{(\ref{7})\Longleftrightarrow}
\displaystyle\sum_{i=1}^{m}a_{i,n}+\displaystyle\sum_{i=1}^{m}b_{i,n}+d_{n}+t_{n}+s_{n}\xrightarrow{\qquad}-\displaystyle\sum_{i=1}^{m}\lambda_{i}f_{i}(\overline{x})-\displaystyle\sum_{i=1}^{m}\lambda_{i}\nu_{i}(-g_{i})(\overline{x})-\delta_{C}^{}(\overline{x})\\
\qquad\qquad\qquad\qquad\qquad\qquad\qquad\qquad\qquad\qquad\qquad-\delta_{-Y_{+}}(h(\overline{x}))\\
=-\displaystyle\sum_{i=1}^{m}\lambda_{i}(f_{i}(\overline{x})+\nu_{i}(-g_{i})(\overline{x}))-\delta_{C}^{}(\overline{x})-\delta_{-Y_{+}}(h(\overline{x})).
\end{multline*}
Since $f_{i}(\overline{x})+\nu_{i}(-g_{i})(\overline{x})=0,$  $\overline{x}\in C$ and $h(\overline{x})\in-Y_{+}$ then the above limit reduces to
\[\displaystyle\sum_{i=1}^{m}a_{i,n}+\displaystyle\sum_{i=1}^{m}b_{i,n}+d_{n}+t_{n}+s_{n}\xrightarrow{\qquad}0.\]
Moreover, (\ref{r1}) can be written as follows
\[\displaystyle\sum_{i=1}^{m}x_{i,n}^{*}+\displaystyle\sum_{i=1}^{m}w_{i,n}^{*}+c_{n}^{*}+u_{n}^{*}\xrightarrow{{\parallel.\parallel_{X^{*}}}}0. \]
This completes the proof.
\end{proof}
For the second characterization, by applying Theorem \ref{4.2} and Theorem \ref{jekumyaar} we express the sequential optimality conditions of the problem $(P)$ in terms of limits for the $\epsilon$-subdifferential $(\epsilon\geq 0)$ of the functions involved at the minimizer.
\begin{theorem}\label{4.3}
Let $\overline{x}\in C\cap h^{-1}(-Y_{+})$ and $\nu_{i}:=\frac{f_{i}(\overline{x})}{g_{i}(\overline{x})}\;(i=1,\hdots,m).$ Then, $\overline{x}$ is Henig properly efficient solution for problem (P), if and only if, there exist $(\lambda_{1},\hdots,\lambda_{m})\in(\mathbb{R}_{+}\setminus\{0\})^{m},$
$\gamma_{n}\geq 0,\;x^{*}_{i,n}\in\partial_{\gamma_{n}}(\lambda_{i}f_{i})(\overline{x}),\;w^{*}_{i,n}\in\partial_{\gamma_{n}}(\lambda_{i}\nu_{i}(-g_{i}))(\overline{x})\;(i=1,\hdots,m),\;c_{n}^{*}\in N_{\gamma_{n}}(\overline{x},C),\;$ $v_{n}^{*}\in -Y_{+}^{*},\;y_{n}^{*}\in Y_{+}^{*}\cap N_{\gamma_{n}}(h(\overline{x}),-Y_{+})$  and $u_{n}^{*}\in\partial_{\gamma_{n}}(-v_{n}^{*}\circ h)(\overline{x})$ such that
\begin{numcases}{}
\gamma_{n}\longrightarrow0\nonumber\\
\displaystyle\sum_{i=1}^{m}x_{i,n}^{*}+\displaystyle\sum_{i=1}^{m}w_{i,n}^{*}+c_{n}^{*}+u_{n}^{*}\xrightarrow{{\parallel.\parallel_{X^{*}}}}0\nonumber\\
y_{n}^{*}+v_{n}^{*}\xrightarrow{{\parallel.\parallel_{Y^{*}}}}0.\nonumber
\end{numcases}
\end{theorem}
\begin{proof} By virtue of Theorem \ref{4.2}, $\overline{x}$ is properly efficient solution of the problem $(P)$ in the sense of Henig, if and only if, there exist $(\lambda_{1},\hdots,\lambda_{m})\in(\mathbb{R}_{+}\setminus\{0\})^{m},\; (x^{*}_{i,n},a_{i,n})\in\mathrm{epi}(\lambda_{i}f_{i})^{*},\; (w^{*}_{i,n},b_{i,n})\in\mathrm{epi}(\lambda_{i}\nu_{i}(-g_{i}))^{*}\; (i=1,\hdots m),\;(c^{*}_{n},d_{n})\in\mathrm{epi}\delta_{C}^{*},\; y^{*}_{n}\in Y_{+}^{*},\; s_{n}\in\mathbb{R}_{+},$ ${v}_{n}^{*}\in-Y_{+}^{*}$ and $(u^{*}_{n},t_{n})\in\mathrm{epi}(-v_{n}^{*}\circ h)^{*}$ such that
\begin{numcases}{}
\displaystyle\sum_{i=1}^{m}x_{i,n}^{*}+\displaystyle\sum_{i=1}^{m}w_{i,n}^{*}+c_{n}^{*}+u_{n}^{*}\xrightarrow{{\parallel.\parallel_{X^{*}}}}0\label{8}\\
y_{n}^{*}+v_{n}^{*}\xrightarrow{{\parallel.\parallel_{Y^{*}}}}0\label{9}\\
\displaystyle\sum_{i=1}^{m}a_{i,n}+\displaystyle\sum_{i=1}^{m}b_{i,n}+d_{n}+t_{n}+s_{n}\xrightarrow{\qquad}0.\label{10}
\end{numcases}

Since $(y_{n}^{*},s_{n})\in\mathrm{epi}\delta_{Y_{+}^{*}}=\mathrm{epi}\delta^{*}_{-Y_{+}},$  it follows according to Theorem \ref{jekumyaar}, there exist $\alpha_{i,n},\;\beta_{i,n},\;\eta_{n} ,\;\theta_{n},\;\epsilon_{n}\in \mathbb{R}_{+}$ such that $x^{*}_{i,n}\in\partial_{\alpha_{i,n}}(\lambda_{i}f_{i})(\overline{x}),\;$ $w^{*}_{i,n}\in\partial_{\beta_{i,n}}(\lambda_{i}\nu_{i}(-g_{i}))(\overline{x}),\;c_{n}^{*}\in N_{\eta_{n}}(\overline{x},C),\;y_{n}^{*}\in N_{\theta_{n}}(h(\overline{x}),-Y_{+}),\;u_{n}^{*}\in\partial_{\epsilon_{n}}(-v_{n}^{*}\circ h)(\overline{x})$ and
\begin{numcases}{}
a_{i,n}=\langle x^{*}_{i,n}, \overline{x}\rangle+\alpha_{i,n}-(\lambda_{i}f_{i})(\overline{x}),\;\; i=1,\hdots,m\nonumber\\
b_{i,n}=\langle w^{*}_{i,n}, \overline{x}\rangle+\beta_{i,n}-(\lambda_{i}\nu_{i}(-g_{i}))(\overline{x}),\;\; i=1,\hdots,m\nonumber\\
d_{n}=\langle c^{*}_{n}, \overline{x}\rangle+\eta_{n}\nonumber\\
s_{n}=\langle y^{*}_{n}, h(\overline{x})\rangle+\theta_{n}\nonumber\\
t_{n}=\langle u^{*}_{n}, \overline{x}\rangle+\epsilon_{n}-(-v_{n}^{*}\circ h)(\overline{x}).\nonumber
\end{numcases}

By adding the terms of the above equalities and using the fact that $f_{i}(\overline{x})+\nu_{i}(-g_{i})(\overline{x})=0,$ we obtain
$ $\newpage
\begin{eqnarray*}
  \displaystyle\sum_{i=1}^{m}a_{i,n}+\displaystyle\sum_{i=1}^{m}b_{i,n}+d_{n}+t_{n}+s_{n} &=& \displaystyle\sum_{i=1}^{m}[\langle x^{*}_{i,n}, \overline{x}\rangle+\alpha_{i,n}-(\lambda_{i}f_{i})(\overline{x})]\\
&&+\displaystyle\sum_{i=1}^{m}[\langle w^{*}_{i,n}, \overline{x}\rangle+\beta_{i,n}-(\lambda_{i}\nu_{i}(-g_{i}))(\overline{x})]\\
& &+\langle c^{*}_{n}, \overline{x}\rangle+\eta_{n}+\langle y^{*}_{n}, h(\overline{x})\rangle+\theta_{n}+\langle u^{*}_{n}, \overline{x}\rangle+\epsilon_{n}\\
&&-(-v_{n}^{*}\circ h)(\overline{x}).\\
&=&\langle\displaystyle\sum_{i=1}^{m} x^{*}_{i,n}+\displaystyle\sum_{i=1}^{m} w^{*}_{i,n}+c^{*}_{n}+u^{*}_{n}, \overline{x}\rangle+\langle y^{*}_{n}+v_{n}^{*}, h(\overline{x})\rangle\\
&&+\displaystyle\sum_{i=1}^{m}\alpha_{i,n}+\displaystyle\sum_{i=1}^{m}\beta_{i,n}+\eta_{n}+\theta_{n}+\epsilon_{n}.
\end{eqnarray*}
It follows from (\ref{8}), (\ref{9}) and (\ref{10}) that
\begin{equation}\label{11}
\displaystyle\sum_{i=1}^{m}\alpha_{i,n}+\displaystyle\sum_{i=1}^{m}\beta_{i,n}+\eta_{n}+\epsilon_{n}\longrightarrow0,\quad n\longmapsto+\infty.
\end{equation}
Moreover, since $\alpha_{i,n},$ $\beta_{i,n},$ $\eta_{n},$ $\theta_{n},$ $\epsilon_{n}\in\mathbb{R}_{+},$ we get from (\ref{11}) that $\alpha_{i,n}\longrightarrow0,$ $\beta_{i,n}\longrightarrow0,$ $\eta_{n}\longrightarrow0,$
$\theta_{n}\longrightarrow0,$
$\epsilon_{n}\longrightarrow0$ as $n\longmapsto+\infty.$ By setting $\gamma_{n}:=\max_{1\leq i\leq m}\{\alpha_{i,n}, \beta_{i,n}, \eta_{n}, \theta_{n}, \epsilon_{n}\},$ it follows that
$x^{*}_{i,n}\in\partial_{\gamma_{n}}(\lambda_{i}f_{i})(\overline{x}),\; w^{*}_{i,n}\in\partial_{\gamma_{n}}(\lambda_{i}\nu_{i}(-g_{i}))(\overline{x}),\;
c_{n}^{*}\in N_{\gamma_{n}}(\overline{x},C),\; N_{\gamma_{n}}(h(\overline{x}),-Y_{+}),\; u_{n}^{*}\in\partial_{\gamma_{n}}(-v_{n}^{*}\circ h)(\overline{x})$ with $\gamma_{n}\longrightarrow0,$ as $n\longmapsto+\infty.$\\
\bigskip
Conversely, suppose that there exist $(\lambda_{1},\hdots,\lambda_{m})\in(\mathbb{R}_{+}\setminus\{0\})^{m},$
$\gamma_{n}\geq 0,\;x^{*}_{i,n}\in\partial_{\gamma_{n}}(\lambda_{i}f_{i})(\overline{x}),\;w^{*}_{i,n}\in\partial_{\gamma_{n}}(\lambda_{i}\nu_{i}(-g_{i}))(\overline{x})\;(i=1,\hdots,m),\;c_{n}^{*}\in N_{\gamma_{n}}(\overline{x},C),\;
$ $v_{n}^{*}\in -Y_{+}^{*},$ $y_{n}^{*}\in Y_{+}^{*}\cap N_{\gamma_{n}}(h(\overline{x}),-Y_{+}),$  $u_{n}^{*}\in\partial_{\gamma_{n}}(-v_{n}^{*}\circ h)(\overline{x})$ and
\begin{numcases}{}
\gamma_{n}\longrightarrow0\nonumber\\
\displaystyle\sum_{i=1}^{m}x_{i,n}^{*}+\displaystyle\sum_{i=1}^{m}w_{i,n}^{*}+c_{n}^{*}+u_{n}^{*}\xrightarrow{{\parallel.\parallel_{X^{*}}}}0\nonumber\\
y_{n}^{*}+v_{n}^{*}\xrightarrow{{\parallel.\parallel_{Y^{*}}}}0.\nonumber
\end{numcases}

For any $x\in C\cap h^{-1}(-Y_{+}),$ $y\in -Y_{+}$ and any
positive integer $n,$ we have
\begin{eqnarray*}
  \lambda_{i}f_{i}(x)-\lambda_{i}f_{i}(\overline{x}) &\geq& \langle x^{*}_{i,n},x-\overline{x}\rangle-\gamma_{n},\;\; (i=1,\hdots,m)\\
   \lambda_{i}\nu_{i}(-g_{i})(x)-\lambda_{i}\nu_{i}(-g_{i})(\overline{x})&\geq&\langle w^{*}_{i,n},x-\overline{x}\rangle-\gamma_{n},\;\; (i=1,\hdots,m) \\
  0 &\geq&\langle c^{*}_{n},x-\overline{x}\rangle-\gamma_{n} \\
  0 &\geq&\langle y^{*}_{n},y-h(\overline{x})\rangle-\gamma_{n} \\
(-v_{n}^{*}\circ h)(x)-(-v_{n}^{*}\circ h)(\overline{x})&\geq&\langle u^{*}_{n},x-\overline{x}\rangle-\gamma_{n}.
\end{eqnarray*}
By adding them up and by taking $y:=h(x),$ we get
\begin{multline*}
  \displaystyle\sum_{i=1}^{m}\lambda_{i}(f_{i}+\nu_{i}(-g_{i}))(x)-\displaystyle\sum_{i=1}^{m}\lambda_{i}(f_{i}+\nu_{i}(-g_{i}))(\overline{x}) +\langle v_{n}^{*}+y_{n}^{*},h(\overline{x})-h(x) \rangle \\ \geq \langle \displaystyle\sum_{i=1}^{m}x^{*}_{i,n}+\displaystyle\sum_{i=1}^{m}w^{*}_{i,n}+c^{*}_{n}+u^{*}_{n},x-\overline{x}\rangle-5\gamma_{n}.
  \end{multline*}
  By taking the limit as $n\longmapsto+\infty$ in both terms of the above inequality, we deduce that
\[
  \displaystyle\sum_{i=1}^{m}\lambda_{i}(f_{i}+\nu_{i}(-g_{i}))(x)-\displaystyle\sum_{i=1}^{m}\lambda_{i}(f_{i}+\nu_{i}(-g_{i}))(\overline{x})\geq 0,\;\forall x\in C\cap h^{-1}(-Y_{+})
  \]
i.e.
\begin{multline*}
  \displaystyle\sum_{i=1}^{m}\lambda_{i}(f_{i}+\nu_{i}(-g_{i}))(x)+\delta_{C}(x)
  +(\delta_{-Y_{+}}\circ h)(x)-\displaystyle\sum_{i=1}^{m}\lambda_{i}(f_{i}+\nu_{i}(-g_{i}))(\overline{x})-\delta_{C}(\bar{x})\\-\delta_{-Y_{+}}\circ h(\bar{x}) \geq 0,\; \forall x\in X.
  \end{multline*}

  By setting $z^{*}:=(\lambda_{1},\hdots,\lambda_{m}),$ it is clear that $z^{*}\in(\mathbb{R}_{+}\backslash\{0\})^{m}=((\mathbb{R}_{+}^{m})^{*})^{\circ}.$ As $z^{*}\circ\delta_{C}^{v}=\delta_{C}$ and $z^{*}\circ\delta_{-Y_{+}}^{v}=\delta_{-Y_{+}},$ it follows that
  \[0\in \partial\left(z^{*}\circ (f_{1}+\nu_{1}(-g_{1}),\hdots,f_{m}+\nu_{m}(-g_{m}))+z^{*}\circ\delta_{C}^{v}+z^{*}\circ\delta_{-Y_{+}}^{v}\circ h\right)(\overline{x})\]

i.e.
\[0\in \partial\left(z^{*}\circ\left( (f_{1}-\nu_{1}g_{1},\hdots,f_{m}-\nu_{m}g_{m})+\delta_{C}^{v}+\delta^{v}_{-Y_{+}}\circ h\right)\right)(\overline{x}).\]
The mapping $\left((f_{1}-\nu_{1}g_{1},\hdots,f_{m}-\nu_{m}g_{m})+\delta_{C}^{v}+\delta^{v}_{-Y_{+}}\circ h\right)$ is obviously $\mathbb{R}_{+}^{m}$-convex and by virtue of scalarization Theorem \ref{scalarization}, we get
\[0\in\partial^{p}\left( (f_{1}-\nu_{1}g_{1},\hdots,f_{m}-\nu_{m}g_{m})+\delta_{C}^{v}+\delta^{v}_{-Y_{+}}\circ h\right)(\overline{x}),\]
which means that $\overline{x}$ is Henig properly efficient solution for problem $(P_{\overline{x}})$ and by Lemma \ref{l61}, $\overline{x}$ is Henig properly efficient solution for problem $(P)$. The proof is complete.
\end{proof}

Similarly, we establish the sequential optimality
conditions for $(P)$ in terms of the subdifferentials of the functions involved.

\begin{theorem}\label{4.4}
Let $\overline{x}\in C\cap h^{-1}(-Y_{+})$ and $\nu_{i}:=\frac{f_{i}(\overline{x})}{g_{i}(\overline{x})}\; (i=1,\hdots,m).$ Then, $\overline{x}$ is Henig properly efficient solution for problem $(P),$ if and only if, there exist $(\lambda_{1},\hdots,\lambda_{m})\in(\mathbb{R}_{+}\setminus\{0\})^{m},$ $x_{i,n}\in X,\;w_{i,n}\in X,\;c_{n}\in C,\;u_{n}\in\mathrm{dom}(-v_{n}^{*}\circ h),\;y_{n}\in-Y_{+},\;x^{*}_{i,n}\in\partial_{}(\lambda_{i}f_{i})(x_{i,n}),\;w^{*}_{i,n}\in\partial_{}(\lambda_{i}\nu_{i}(-g_{i}))(w_{i,n}),\;c_{n}^{*}\in N_{}(c_{n},C),\;{y}_{n}^{*}\in N_{}(y_{n},-Y_{+}),\;v_{n}^{*}\in -Y_{+}^{*}$ and $u_{n}^{*}\in\partial_{}(-v_{n}^{*}\circ h)(u_{n})$ such that
\begin{numcases}{}
x_{i,n}\xrightarrow{\parallel.\parallel_{X}}\overline{x},\ w_{i,n}\xrightarrow{\parallel.\parallel_{X}}\overline{x},\ c_{n}\xrightarrow{\parallel.\parallel_{X}}\overline{x},\ u_{n} \xrightarrow{\parallel.\parallel_{X}}\overline{x},\;\nonumber\\
\displaystyle\sum_{i=1}^{m}x^{*}_{i,n}+\displaystyle\sum_{i=1}^{m}w^{*}_{i,n}+c_{n}^{*}+u^{*}_{n} \xrightarrow{\|.\|_{X^{*}}}0,\ y^{*}_{n}+v^{*}_{n} \xrightarrow{\|.\|_{Y^{*}}}0,\nonumber\\
\lambda_{i}f_{i}(x_{i,n})-\langle x^{*}_{i,n}, x_{i,n}-\overline{x}\rangle\longrightarrow \lambda_{i}f_{i}(\overline{x})\;(i=1,\hdots,m)\quad\nonumber\\
\lambda_{i}\nu_{i}(-g_{i})(w_{i,n})-\langle w^{*}_{i,n}, w_{i,n}-\overline{x}\rangle\longrightarrow \lambda_{i}\nu_{i}(-g_{i})(\overline{x})\quad(i=1,\hdots,m)\nonumber\\
\langle c^{*}_{n},c_{n}-\overline{x}\rangle\longrightarrow 0,\;\langle y_{n}^*,y_{n}-h(\bar x)\rangle{\longrightarrow} 0\nonumber\\
\langle u^{*}_{n},u_{n}-\overline{x}\rangle+\langle v_{n}^{*},h(u_{n})-h(\overline{x})\rangle\longrightarrow 0.\nonumber
\end{numcases}

\end{theorem}
\begin{proof}
According to Theorem \ref{4.3}, $\overline{x}$ is Henig properly efficient solution for problem $(P),$ if and only if, there exist $(\lambda_{1},\hdots,\lambda_{m})\in(\mathbb{R}_{+}\setminus\{0\})^{m},$
$\gamma_{n}\geq 0,\; \overline{x}^{*}_{i,n}\in\partial_{\gamma_{n}}(\lambda_{i}f_{i})(\overline{x}),\; \overline{w}^{*}_{i,n}\in\partial_{\gamma_{n}}(\lambda_{i}\nu_{i}(-g_{i}))(\overline{x}),\;\overline{c}_{n}^{*}\in N_{\gamma_{n}}(\overline{x},C),\; {v}_{n}^{*}\in -Y_{+}^{*},\;\overline{y}_{n}^{*}\in Y_{+}^{*}\cap N_{\gamma_{n}}(h(\overline{x}),-Y_{+}),\; \overline{u}_{n}^{*}\in\partial_{\gamma_{n}}(-v_{n}^{*}\circ h)(\overline{x})$ and
\begin{numcases}{}
\gamma_{n}\longrightarrow0\nonumber\\
\overline{y}_{n}^{*}+v_{n}^{*}\xrightarrow{{\parallel.\parallel_{Y^{*}}}}0
\label{in1}\\
\displaystyle\sum_{i=1}^{m}\overline{x}_{i,n}^{*}+\displaystyle\sum_{i=1}^{m}\overline{w}_{i,n}^{*}+\overline{c}_{n}^{*}+\overline{u}_{n}^{*}\xrightarrow{{\parallel.\parallel_{X^{*}}}}0\label{in2}.
\end{numcases}
As
$\overline{x}^{*}_{i,n}\in\partial_{\gamma_{n}}(\lambda_{i}f_{i})(\overline{x}),$ $\overline{w}^{*}_{i,n}\in\partial_{\gamma_{n}}(\lambda_{i}\nu_{i}(-g_{i}))(\overline{x}),$
$\overline{c}_{n}^{*}\in N_{\gamma_{n}}(\overline{x},C),$ $\overline{y}_{n}^{*}\in N_{\gamma_{n}}(h(\overline{x}),-Y_{+}),$ $\overline{u}_{n}^{*}\in\partial_{\gamma_{n}}(-v_{n}^{*}\circ h)(\overline{x})$ then by applying Theorem \ref{thibault}, we obtain the existence of  ${x}_{i,n}\in\mathrm{dom}(\lambda_{i}f_{i})=X,$
${w}_{i,n}\in\mathrm{dom}(\lambda_{i}\nu_{i}(-g_{i}))=X,$
${c}_{n}\in C,$ $y_{n}\in -Y_{+},$
${u}_{n}\in\mathrm{dom}(-v_{n}^{*}\circ h),$
${x}^{*}_{i,n}\in\partial_{}(\lambda_{i}f_{i})({x_{i,n}}),$
${w}^{*}_{i,n}\in\partial_{}(\lambda_{i}\nu_{i}(-g_{i}))({w_{i,n}}),$
${c}_{n}^{*}\in N_{}({c}_{n},C),$  ${y}_{n}^{*}\in N_{}(y_{n},-Y_{+}),$ ${u}_{n}^{*}\in\partial(-v_{n}^{*}\circ h)(u_{n})$ such that

\begin{numcases}{}
\begin{multlined}
\|x_{i,n} - \bar x\|_X \leq\sqrt{\gamma_{n}},  \;   \|w_{i,n} - \bar x\|_X \leq\sqrt{\gamma_{n}},  \;  \|c_{n} - \bar x\|_X \leq\sqrt{\gamma_{n}}, \; \|u_{n} - \bar x\|_X \leq\sqrt{\gamma_{n}}\\
\|y_{n} - h(\bar x)\|_Y \leq\sqrt{\gamma_{n}},\qquad\quad\qquad\quad\qquad\quad\qquad\quad\qquad\quad\qquad\quad\qquad\quad\qquad\quad
\end{multlined}\label{13}\\
\begin{multlined}
\|x^{*}_{i,n} - \bar x^{*}_{i,n}\|_{X^*} \leq\sqrt{\gamma_{n}},\; \|w^{*}_{i,n} - \bar w^{*}_{i,n}\|_{X^*} \leq\sqrt{\gamma_{n}},  \|c^{*}_{n} - \bar c^{*}_{n}\|_{X^*} \leq\sqrt{\gamma_{n}}\\
\|u^{*}_{n} - \bar u^{*}_{n}\|_{X^*} \leq\sqrt{\gamma_{n}}, \; \|y^{*}_{n} - \bar y^{*}_{n}\|_{Y^*} \leq\sqrt{\gamma_{n}},
\qquad\quad\qquad\quad\qquad\quad\qquad\quad
\end{multlined}\label{14}\\
\left.
    \begin{array}{ll}
\mid\lambda_{i}f_{i}(x_{i,n})-\langle x_{i,n}^*,x_{i,n}-\bar x\rangle -\lambda_{i} f_{i}(\bar x)\mid\leq2\gamma_{n}\\
\mid-\lambda_{i}\nu_{i}g_{i}({w}_{i,n})-\langle w_{i,n}^*,w_{i,n}-\bar x\rangle+\lambda_{i}\nu_{i}g_{i}(\bar x)\mid\leq2\gamma_{n}\\
\mid\delta^{}_{C}(c_{n})-\langle c_{n}^*,c_{n}-\bar x\rangle - \delta^{}_{C}(\bar x)\mid\leq2\gamma_{n}\label{1a1}\\
\mid\delta^{}_{-Y_{+}}(y_{n})-\langle y_{n}^*,y_{n}-h(\bar x)\rangle - \delta^{}_{-Y_{+}}(h(\bar x))\mid\leq2\gamma_{n}\\
\mid(-v_{n}^{*}\circ h)(u_{n})-\langle u_{n}^{*},u_{n}-\bar x\rangle-(-v_{n}^{*}\circ h)(\bar x)\mid\leq 2\gamma_{n}
\end{array}
\right\}.\label{15}
\end{numcases}

By letting $n\longmapsto+\infty,$ we get from (\ref{13}) and (\ref{15}) that

\begin{numcases}{}
x_{i,n} \stackrel {\|.\|_{X}}{\longrightarrow} \bar x,\  w_{i,n} \stackrel {\|.\|_{X}}{\longrightarrow} \bar x,\  c_{n} \stackrel {\|.\|_{X}}{\longrightarrow} \bar x, u_{n} \stackrel {\|.\|_{X}}{\longrightarrow} \bar x,\ y_{n} \stackrel {\|.\|_{Y}}{\longrightarrow} h(\bar x)\nonumber\\
\lambda_{i}f_{i}(x_{i,n})-\langle x_{i,n}^*,x_{i,n}-\bar x\rangle \stackrel {}{\longrightarrow} \lambda_{i} f_{i}(\bar x),\nonumber\\
-\lambda_{i}\nu_{i}g_{i}({w}_{i,n})-\langle w_{i,n}^*,w_{i,n}-\bar x\rangle \stackrel {}{\longrightarrow} -\lambda_{i}\nu_{i}g_{i}(\bar x),\nonumber\\
\left.
    \begin{array}{ll}
\delta^{}_{C}(c_{n})-\langle c_{n}^*,c_{n}-\bar x\rangle \stackrel {}{\longrightarrow} \delta^{}_{C}(\bar x).\\
\delta^{}_{-Y_{+}}(y_{n})-\langle y_{n}^*,y_{n}-h(\bar x)\rangle{\longrightarrow} \delta^{}_{-Y_{+}}(h(\bar x))\end{array}
\right\}\label{16}\\
(-v_{n}^{*}\circ h)(u_{n})-\langle u_{n}^{*},u_{n}-\bar x\rangle-(-v_{n}^{*}\circ h)(\bar x) {\longrightarrow}0.\nonumber
\end{numcases}

Since $c_{n}\in C,\; y_{n}\in-Y_{+}$ and $\overline{x}\in C\cap h^{-1}(-Y_+),$ the expression (\ref{16}) reduces to
\[\langle c_{n}^*,c_{n}-\bar x\rangle \stackrel {}{\longrightarrow} 0,\quad
\langle y_{n}^*,y_{n}-h(\bar x)\rangle{\longrightarrow} 0.\]

Moreover, we have

\begin{eqnarray*}
  \|y^{*}_{n}+v^{*}_{n}\|_{Y^*} &=& \| y^{*}_{n}-\overline{y}_{n}^{*}+v^{*}_{n}+\overline{y}_{n}^{*}\|_{Y^*} \leq\| y^{*}_{n}-\overline{y}_{n}^{*}\|_{Y^*}+\|\overline{y}_{n}^{*}+{v}_{n}^{*} \|_{Y^*},
\end{eqnarray*}
and
$ $\newpage
\begin{eqnarray*}
  && \|\displaystyle \sum_{i=1}^{m}x^{*}_{i,n}+\displaystyle \sum_{i=1}^{m}w^{*}_{i,n}+c_{n}^{*}+u^{*}_{n}\|_{X^*} \\
   &&=\|\displaystyle \sum_{i=1}^{m}x^{*}_{i,n}-\displaystyle\sum_{i=1}^{m}\overline{x}_{i,n}^{*}+\displaystyle \sum_{i=1}^{m}w^{*}_{i,n}-\displaystyle \sum_{i=1}^{m}\overline{w}^{*}_{i,n}+c^{*}_{n}-\overline{c}^{*}_{n}+u^{*}_{n}- \overline{u}_{n}^{*}       +\displaystyle\sum_{i=1}^{m}\overline{x}_{i,n}^{*}\\
   &&\qquad\qquad\qquad\qquad+\displaystyle\sum_{i=1}^{m}\overline{w}_{i,n}^{*}+\overline{c}_{n}^{*}+\overline{u}_{n}^{*}\|_{X^*} \\
   &&\leq  \displaystyle\sum_{i=1}^{m}\|\displaystyle x^{*}_{i,n}-\overline{x}_{i,n}^{*}\|_{X^*}+\displaystyle\sum_{i=1}^{m}\|\displaystyle w^{*}_{i,n}-\overline{w}_{i,n}^{*}\|_{X^*}+\|c^{*}_{n}-\overline{c}_{n}^{*}\|_{X^*}+\|u^{*}_{n}-\overline{u}_{n}^{*}\|_{X^*}      \\
   &&\qquad\qquad+\|\displaystyle\sum_{i=1}^{m}\overline{x}_{i,n}^{*}+\displaystyle\sum_{i=1}^{m}\overline{w}_{i,n}^{*}+\overline{c}^{*}_{n}+\overline{u}_{n}^{*}\|_{X^*}.\\
\end{eqnarray*}

\noindent and hence by letting $n\longrightarrow +\infty,$ it follows from (\ref{in1}), (\ref{in2}) and (\ref{14}) that

\begin{numcases}{}
 y^{*}_{n}+v^{*}_{n} \xrightarrow{\|.\|_{Y^{*}}}0,\nonumber\\
\displaystyle\sum_{i=1}^{m}x^{*}_{i,n}+\displaystyle\sum_{i=1}^{m}w^{*}_{i,n}+c_{n}^{*}+u^{*}_{n} \xrightarrow{\|.\|_{X^{*}}}0.\nonumber
\end{numcases}
Conversely, assume that the preceding sequential optimality conditions hold. Then, for any $x\in C\cap h^{-1}(-Y_{+}),$ $y\in -Y_{+}$ and any
positive integer $n,$ we have
\begin{eqnarray*}
  \lambda_{i}f_{i}(x) &\geq& \lambda_{i}f_{i}(x_{i,n})+\langle x^{*}_{i,n},x-x_{i,n}\rangle,\;\; (i=1,\hdots,m)\\
   \lambda_{i}\nu_{i}(-g_{i})(x)&\geq&\lambda_{i}\nu_{i}(-g_{i})(w_{i,n})+\langle w^{*}_{i,n},x-w_{i,n}\rangle,\;\; (i=1,\hdots,m) \\
  0 &\geq&\langle c^{*}_{n},x-c_{n}\rangle \\
  0 &\geq&\langle y^{*}_{n},y-y_{n}\rangle \\
0&\geq&-(-v_{n}^{*}\circ h)(x)+(-v_{n}^{*}\circ h)(u_{n})+\langle u^{*}_{n},x-u_{n}\rangle.
\end{eqnarray*}
By adding these inequalities and taking for all $x\in C\cap h^{-1}(-Y_{+}),\;y:=h(x)$, we get
\begin{eqnarray*}
\displaystyle\sum_{i=1}^{m}\lambda_{i}f_{i}(x)+\displaystyle\sum_{i=1}^{m}\lambda_{i}\nu_{i}(-g_{i})(x)&\geq& \displaystyle\sum_{i=1}^{m}\left(\lambda_{i}f_{i}(x_{i,n})-\langle x^{*}_{i,n},x_{i,n}-\overline{x}\rangle\right)\\
&&+\displaystyle\sum_{i=1}^{m}\left(\lambda_{i}\nu_{i}(-g_{i})(w_{i,n})-\langle w^{*}_{i,n},w_{i,n}-\overline{x}\rangle\right)\\
&&-\langle c^{*}_{n},c_{n}-\overline{x}\rangle-\langle y^{*}_{n},y_{n}-h(\overline{x})\rangle-\langle u^{*}_{n},u_{n}-\overline{x}\rangle\\
&&-\langle v_{n}^{*},h(u^{*}_{n})-h(\overline{x})\rangle+\langle y^{*}_{n}+v_{n}^{*},h(x)-h(\overline{x})\rangle\\
&&+\langle \displaystyle\sum_{i=1}^{m} x^{*}_{i,n}+\displaystyle\sum_{i=1}^{m} w^{*}_{i,n}+c_{n}^{*}+u^{*}_{n},x-\overline{x}\rangle.
\end{eqnarray*}

Passing to the limit as $n\longmapsto+\infty$, we obtain
\[
  \displaystyle\sum_{i=1}^{m}\lambda_{i}(f_{i}+\nu_{i}(-g_{i}))(x)-\displaystyle\sum_{i=1}^{m}\lambda_{i}(f_{i}+\nu_{i}(-g_{i}))(\overline{x})\geq 0,\;\forall x\in C\cap h^{-1}(-Y_{+})
  \]
i.e.
\begin{multline*}
  \displaystyle\sum_{i=1}^{m}\lambda_{i}(f_{i}+\nu_{i}(-g_{i}))(x)+\delta_{C}(x)+\delta_{-Y_{+}}\circ h(x)-\displaystyle\sum_{i=1}^{m}\lambda_{i}(f_{i}+\nu_{i}(-g_{i}))(\overline{x})-\delta_{C}(\bar{x})\\-\delta_{-Y_{+}}\circ h(\bar{x}) \geq 0,\; \forall x\in X.
  \end{multline*}

By using the similar arguments used in the proof of the converse of Theorem \ref{4.3}, we get
\[0\in\partial^{p}\left( (f_{1}-\nu_{1}g_{1},\hdots,f_{m}-\nu_{m}g_{m})+\delta_{C}^{v}+\delta^{v}_{-Y_{+}}\circ h\right)(\overline{x}),\]
which means that $\overline{x}$ is Henig properly efficient solution for problem $(P_{\overline{x}})$ and by Lemma \ref{l61}, $\overline{x}$ is Henig properly efficient solution for problem $(P)$. The proof is complete.
\end{proof}
We now illustrate the above results with the help of an example of multiobjective fractional programming problem, where the standard Lagrange multiplier condition can not be derived due to the lack of constraint qualification and the sequential conditions hold. For this, we will need to establish the standard necessary and sufficient optimality conditions for a feasible point $\overline{x}$ to be an efficient solution for problem (P) under a constraint qualification.

\begin{theorem}\label{optimalite classic}
Let $f_{i}, \; \mathrm{-}g_{i}$ : $X\longrightarrow\mathbb{R}$ be $2m$ convex functions such that $f_{i}(x)\geq0$ for any $x\in C\cap h^{-1}(-Y_{+})$ $(i=1,\hdots,m)$ and $h$ : $X\longrightarrow Y\cup\{+\infty_{Y}\}$ be a proper and $Y_{+}$-convex mapping. Let us consider the following constraint qualification

\begin{empheq}[left= (C.Q.M_{0}.R_{0})\empheqlbrace]{align*}
& (X,\|.\|_X)\  \text{and}\ (Y,\|.\|_Y) \text{ are two real reflexive
Banach spaces,}\\
& \exists a\in C\cap \mathrm{dom}h\ \text{such that} \\
& h(a)\in-\mathrm{int}Y_{+}.
\end{empheq}

\noindent Suppose that $\mathrm{int}Y_{+}\neq \emptyset$ ($\mathrm{int}Y_{+}$ stands for the topological interior of $Y_{+}$) and the constraint qualification $(C.Q.M_{0}.R_{0})$ is satisfied. Then $\overline{x}\in C\cap h^{-1}(-Y_{+})$ is
Henig properly efficient solution for problem (P), if and only if, there exists $y^{*}\in Y_{+}^{*}$ such that $\langle y^{*}, h(\overline{x})\rangle=0$ and
\[ 0\in \partial(\displaystyle\sum_{i=1}^{m}\left(\lambda_{i}(f_{i}-\nu_{i}g_{i}) +\delta_{C}+y^{*}\circ h\right)(\overline{x}).\]
\end{theorem}

\begin{proof}
Following the proof of Theorem \ref{4.2}, we have $\overline{x}$ is Henig properly efficient solution of $(P)$ if and only if,
\begin{equation}
0\in\partial^{p}\left(F_{\overline{x}}+\delta^{v}_{C}+\delta^{v}_{-Y_{+}}\circ h\right)(\overline{x})
\end{equation}
\noindent where $F_{\overline{x}}\ :\ X\longrightarrow\mathbb{R}^{m}$ is defined for any $x\in X$ by
 \[F_{\overline{x}}(x):=\left(f_{1}(x)-\nu_{1}g_{1}(x),\hdots,f_{m}(x)-\nu_{m}g_{m}(x)\right).\]
According to scalarization Theorem \ref{scalarization}, there exists $z^{*}=(\lambda_{1},\hdots,\lambda_{m})\in
(\mathbb{R}_{+}\setminus\{0\})^{m}$ such that

\begin{equation}\label{19}
   0\in\partial_{}(z^{*}\circ F_{\overline{x}}+z^{*}\circ \delta^{v}_{C}+z^{*}\circ
   \delta^{v}_{-Y_{+}}\circ h)(\overline{x}).
\end{equation}

\noindent Since  $z^{*}\circ \delta_{C}^{v}=\delta_{C}$ and $z^{*}\circ
\delta_{-Y_{+}}^{v}=\delta_{-Y_{+}},$  therefore we obtain

\[ 0\in \partial(\displaystyle\sum_{i=1}^{m}\left(\lambda_{i}(f_{i}-\nu_{i}g_{i}) +\delta_{C}+\delta_{-Y_{+}}\circ h\right)(\overline{x}).\]
\noindent The constraint qualification $(C.Q.M_{0}.R_{0})$ show that the indicator function $\delta_{-Y_{+}}$ is continuous at $h(a)$ and by applying a formula in \cite{combari} by Combari-Laghdir-Thibault, concerning the computation of the subdifferential of the composite of a nondecreasing convex function with a convex mapping taking values in a partially ordered topological vector space, there exists $y^{*}\in \partial\delta_{-Y_{+}}(h(\overline{x}))=N(h(\overline{x}),-Y_{+})$ such that
\[0\in \partial(\displaystyle\sum_{i=1}^{m}\left(\lambda_{i}(f_{i}-\nu_{i}g_{i}) +\delta_{C}+y^{*}\circ h\right)(\overline{x}).\]

\noindent The condition $y^{*}\in N(h(\overline{x}),-Y_{+})$ is equivalent to
$y^{*}\in Y^{*}_{+}$ and $\langle y^{*}, h(\overline{x})\rangle=0.$
\end{proof}
We now give an example of multiobjective fractional programming problem, where the standard optimality condition can not be derived due to the lack of constraint
qualification and the sequential optimality conditions hold.
\begin{example}Let us consider the following multiobjective fractional problem\\

\begin{empheq}[left=(\mathrm{Q})\empheqlbrace]{align*}
&  \displaystyle \inf \left(\frac{2x}{y+3},\frac{2x}{y^{2}+1}\right)\\
& \left(\left(\max{\left\{0,x\right\}}\right)^{2},\sqrt{x^{2}+y^{2}}-y\right)\in-\mathbb{R}_{+}^{2}\\
& (x,y)\in\mathbb{R}_{+}\times[0,1] ,
\end{empheq}
 where $h(x,y)=\left(\left(\max{\left\{0,x\right\}}\right)^{2},\sqrt{x^{2}+y^{2}}-y\right),$ $f_{1}(x,y)=2x,$ $f_{2}(x,y)=-2x,$ $g_{1}(x,y)$ $=y+3,$ $g_{2}(x,y)=-y^{2}-1$ and $C=\mathbb{R}_{+}\times[0,1].$ The euclidian space $Y = \mathbb{R}^{2}$ is equipped with the natural order induced by the nonnegative orthant $Y_{+} = \mathbb{R}^{2}_{+}.$ Obviously $Y_{+}^{*}=\mathbb{R}_{+}^{2}.$
 Let $(\overline{x},\overline{y})=\left(0,\frac{1}{2}\right)$ be a feasible point and $\nu_{i}=\frac{f_{i}(\overline{x},\overline{y})}{g_{i}(\overline{x},\overline{y})}\; (i=1,2).$ Then $\nu_{1}=\nu_{2}=0.$ Let us observe that for any $(x,y)\in C,$ we have $h(x,y)\in\mathbb{R}_{+}^{2},$ and hence the feasible set of problem $(Q)$ is given by $S=\{0\} \times[0,1]$ which yields that the constraint qualification $(C.Q.M_{0}.R_{0})$  does not hold. By taking $(\lambda_{1},\lambda_{2}):=(1,1),\;v_{n}^{*}:=(0,0),$ it follows that
the epigraph of the conjugate functions turn out to be
$\mathrm{epi}(\lambda_{1}f_{1})^{*}=\mathrm{epi}f_{1}^{*}=\{(2,0)\}\times[0,+\infty[,$ $\mathrm{epi}(\lambda_{2}f_{2})^{*}=\mathrm{epi}f_{2}^{*}=\{(-2,0)\}\times[0,+\infty[,$ $\mathrm{epi}(\lambda_{i}\nu_{i}(-g_{i}))^{*}=\{(0,0)\}\times[0,\;+\infty[,$  $\mathrm{epi}\delta_{C}^{*}=\displaystyle\bigcup_{\alpha>0}\{((0,\alpha),\alpha)\}\cup\{0\}\times\mathbb{R}\times[0,\;+\infty[,$ $\mathrm{epi}(-v_{n}^{*}\circ h)^{*}=\{(0,0)\}\times[0,+\infty[.$ For $i=1,2,$ by taking $(x^{*}_{1,n},a_{1,n}):=((2,0),\frac{1}{n})\in\mathrm{epi}(\lambda_{1}f_{1})^{*},\; (x^{*}_{2,n},a_{2,n}):=((-2,0),\frac{1}{n})\in\mathrm{epi}(\lambda_{2}f_{2})^{*},\;(w^{*}_{i,n},b_{i,n}):=((0,0),\frac{1}{n})\in\mathrm{epi}(\lambda_{i}\nu_{i}(-g_{i}))^{*},
\;(c^{*}_{n},d_{n}):=((0,\frac{1}{n}),\frac{1}{n})\in\mathrm{epi}\delta_{C}^{*},\\y^{*}_{n}:=(0,0)\in (\mathbb{R}_{+})^{2},\;s_{n}:=\frac{1}{n}\in\mathbb{R}_{+},$ and $(u^{*}_{n},t_{n})=((0,0),0)\in\mathrm{epi}(-v_{n}^{*}\circ h)^{*}$ such that
\begin{numcases}{}
\displaystyle\sum_{i=1}^{2}x_{i,n}^{*}+\displaystyle\sum_{i=1}^{2}w_{i,n}^{*}+c_{n}^{*}+u_{n}^{*}=(0,\frac{1}{n})\xrightarrow{{\parallel.\parallel_{\mathbb{R}^{2}}}}0\nonumber\\
y_{n}^{*}+v_{n}^{*}=(0,0)\xrightarrow{{\parallel.\parallel_{\mathbb{R}^{2}}}}0\nonumber\\
\displaystyle\sum_{i=1}^{2}a_{i,n}+\displaystyle\sum_{i=1}^{2}b_{i,n}+d_{n}+t_{n}+s_{n}=\frac{6}{n}\xrightarrow{\qquad}0.\nonumber
\end{numcases}

Therefore,  by Theorem \ref{4.2}  the point  $(\overline{x},\overline{y})$ is a Henig efficient solution for $(Q).$
\end{example}



\end{document}